\documentclass[graybox]{svmult}

\usepackage{mathptmx}       % selects Times Roman as basic font
\usepackage{helvet}         % selects Helvetica as sans-serif font
\usepackage{courier}        % selects Courier as typewriter font
\usepackage{type1cm}        % activate if the above 3 fonts are
\usepackage{makeidx}         % allows index generation
\usepackage{graphicx}        % standard LaTeX graphics tool
\usepackage{multicol}        % used for the two-column index
\usepackage[bottom]{footmisc}% places footnotes at page bottom
\usepackage{latexsym,amsmath,amssymb,amsfonts}
\makeindex             % used for the subject index
\newcommand{\R}{\mathbb{R}}

\begin{document}

\title*{Fractional variational calculus\\
for non-differentiable functions}

\author{Agnieszka B. Malinowska}

\institute{Agnieszka B. Malinowska \at Faculty of Computer Science, Bia{\l}ystok University of Technology,
15-351 Bia\l ystok, Poland, \email{abmalinowska@ua.pt}}
\maketitle

\abstract*{The fractional calculus of
variations is a subject under strong
research. Different definitions for fractional derivatives and
integrals are used, depending on the purpose under study.
The fractional operators in this paper are defined in the sense of Jumarie. 
This allows us to work with functions which are non-differentiable.
We present necessary and sufficient optimality
conditions for fractional problems of the calculus of variations
with a Lagrangian density depending on the free end-points.}

\abstract{The fractional calculus of
variations is a subject under strong
research. Different definitions for fractional derivatives and
integrals are used, depending on the purpose under study.
The fractional operators in this paper are defined in the sense of Jumarie. 
This allows us to work with functions which are non-differentiable.
We present necessary and sufficient optimality
conditions for fractional problems of the calculus of variations
with a Lagrangian density depending on the free end-points.}

\section{Introduction}
The Fractional Calculus (FC) is one of the most interdisciplinary
fields of mathematics, with many applications in physics and
engineering. The history of FC goes back more than three centuries,
when in 1695 the derivative of order $\alpha=1/2$ was described by
Leibniz. Since then, many different forms of fractional operators
were introduced: the Grunwald--Letnikov, Riemann--Liouville, Riesz,
and Caputo fractional derivatives \cite{Kilbas,Podlubny,samko}), and the more recent notions of
\cite{Cresson,Jumarie1,Klimek,Kolwankar}. FC is nowadays
the realm of physicists and mathematicians, who investigate the
usefulness of such non-integer order derivatives and integrals in
different areas of physics and mathematics (see, \textrm{e.g.},
\cite{Carpinteri,Hilfer,Kilbas}). It is a successful tool for
describing complex quantum field dynamical systems, dissipation, and
long-range phenomena that cannot be well illustrated using ordinary
differential and integral operators (see, \textrm{e.g.},
\cite{El-Nabulsi:Torres:2008,Hilfer,Klimek,rie}). Applications of FC
are found in classical and quantum mechanics, field
theories, variational calculus, and optimal control (see,
\textrm{e.g.},
\cite{El-Nabulsi:Torres:2007,Frederico:Torres2,Jumarie4}).

The calculus of variations is an old branch of optimization theory
that has many applications both in physics and geometry. Apart from
a few examples known since ancient times such as Queen Dido's
problem (reported in {\it The Aeneid} by Virgil), the problem of
finding optimal curves and surfaces has been posed first by
physicists such as Newton, Huygens, and Galileo. Their contemporary
mathematicians, starting with the Bernoulli brothers and Leibniz,
followed by Euler and Lagrange, invented the calculus of variations
of a functional in order to solve those problems. Fractional
Calculus of Variations (FCV) unifies the calculus of variations and
the fractional calculus, by inserting fractional derivatives into
the variational integrals. This occurs naturally in many problems of
physics or mechanics, in order to provide more accurate models of
physical phenomena. The FCV started in 1996 with the work of
\cite{rie}. Riewe formulated the problem of the calculus of
variations with fractional derivatives and obtained the respective
Euler--Lagrange equations, combining both conservative and
nonconservative cases. Nowadays the FCV is a subject under strong
research. Different definitions for fractional derivatives and
integrals are used, depending on the purpose under study.
Investigations cover problems depending on Riemann--Liouville
fractional derivatives (see, \textrm{e.g.},
\cite{Atanackovic,El-Nabulsi:Torres:2008,Frederico:Torres1,Tatiana:Spain2010}), the
Caputo fractional derivative (see, \textrm{e.g.},
\cite{AGRA,Baleanu1,Malinowska,ComCa}), the symmetric fractional
derivative (see, \textrm{e.g.}, \cite{Klimek}), the Jumarie
fractional derivative (see, \textrm{e.g.},
\cite{Almeida,al:ma:tor,Jumarie1,Jumarie4,Jumarie3,Jumarie5,Sidi}),
and others \cite{Ric:Del:09,Cresson,El-Nabulsi:Torres:2007}. For
applications of the fractional calculus of variations we refer the
reader to
\cite{El-Nabulsi:Torres:2008,Jumarie4,Klimek,Rabei2}.
Although the literature of FCV is already vast, much remains to be
done.

In this paper we study problems of FCV which are defined in terms of
the Jumarie fractional derivatives and integrals. The Euler--Lagrange
equations for such problems with and without constraints were
recently shown in \cite{Almeida}. Here we develop further the theory
by proving necessary optimality conditions for more general problems
of FCV with a Lagrangian that may also depend on the unspecified
end-points $y(a)$, $y(b)$. More precisely, the problem under our
study: to extremize a functional which is defined in terms of the
Jumarie fractional operators and having no constraint on $y(a)$
and/or $y(b)$. The novelty is the dependence of the integrand $L$ on
the a priori unknown final values $y(a)$, $y(b)$. The new natural
boundary conditions (\ref{new:bca})--(\ref{new:bcb}) have important
implications in economics (see \cite{Cruz} and the references
therein).

The paper is organized as follows. Section~\ref{sec1} presents the
necessary definitions and concepts of Jumarie's fractional calculus.
Our results are formulated, proved, and illustrated through examples
in Section~\ref{sec2}. Main results of the paper include necessary
optimality conditions with the generalized natural boundary
conditions (Theorem~\ref{Theo E-L1}) that become sufficient under
appropriate convexity assumptions (Theorem~\ref{suff}). We finish
with Section~\ref{sec3} of conclusions.

%-------------------------------------------------------------------------
\section{Fractional Calculus}\label{sec1}
For an introduction to the classical fractional calculus we refer
the reader to \cite{Kilbas,miller,Podlubny,samko}. In this section
we briefly review the main notions and results from the recent
fractional calculus proposed by Jumarie
\cite{Jumarie1,Jumarie4,Jumarie3}. \begin{definition}
Let $f:[a,b]\to\mathbb R$ be a continuous function. The Jumarie fractional derivative of $f$
is defined by
\begin{equation*}
f^{(\alpha)}(t)
:=\frac{1}{\Gamma(-\alpha)}\int_0^t(t-\tau)^{-\alpha-1}(f(\tau)-f(a))\,d\tau,
\quad \alpha<0,
\end{equation*}
where $\Gamma(z)=\int_0^\infty t^{z-1}e^{-t}\, dt$. For positive
$\alpha$, one will set
\begin{equation*}
f^{(\alpha)}(t)=(f^{(\alpha-1)}(t))'
=\frac{1}{\Gamma(1-\alpha)}\frac{d}{dt}\int_0^t(t-\tau)^{-\alpha}(f(\tau)-f(a))\,d\tau,
\end{equation*}
for $0<\alpha<1$, and
\begin{equation*}
f^{(\alpha)}(t):=(f^{(\alpha -n)}(t))^{(n)}, \quad n\leq\alpha<n+1,
\quad n\geq 1.
\end{equation*}
\end{definition}

The Jumarie fractional derivative has the following property:
\begin{itemize}
\item{The $\alpha$th derivative of a constant is zero.}
\item{Assume that $0<\alpha\leq1$, then the Laplace transform\\
of $f^{(\alpha)}$ is
\begin{equation*}
\mathfrak{L}\{f^{(\alpha)}(t)\}=s^{\alpha}\mathfrak{L}\{f(t)\}-s^{\alpha-1}f(0).
\end{equation*}}
\item{
$(g(t)f(t))^{(\alpha)}=g^{(\alpha)}(t)f(t)+g(t)f^{(\alpha)}(t),
\quad 0<\alpha<1$.}
\end{itemize}

\begin{example}
Let $f(t)=t^\gamma$. Then
  $f^{(\alpha)}(x)=\Gamma(\gamma+1)\Gamma^{-1}(\gamma+1-\alpha)t^{\gamma-\alpha}$,
  where $0<\alpha<1$ and $\gamma>0$.
\end{example}
\begin{example}
The solution of the fractional differential equation
  $$x^{(\alpha)}(t)=c, \quad x(0)=x_0, \quad c=constant,$$ is $$
  x(t)=\frac{c}{\alpha!}t^{\alpha}+x_0,$$
  with the notation ${\alpha!}:=\Gamma(1+\alpha)$.
\end{example}

The integral with respect to $(dt)^\alpha$ is defined as the
solution of the fractional differential equation
  \begin{equation}\label{int}
  dy=f(x)(dx)^{\alpha},\quad x\geq0,\quad y(0)=y_0,\quad
  0<\alpha\leq 1
  \end{equation}
which is provided by the following result:
  \begin{lemma}
  \label{integral}
  Let $f(t)$ denote a continuous function. The solution of the
  equation \eqref{int} is defined by the equality
 $$ \int_0^tf(\tau)(d\tau)^\alpha=\alpha\int_0^t(t-\tau)^{\alpha-1}f(\tau)d\tau, \quad
  0<\alpha\leq 1.$$
  \end{lemma}

\begin{example}
 Let $f(t)=1$. Then $\int_0^t(d\tau)^\alpha=t^{\alpha}$,
$0<\alpha\leq 1$.
\end{example}
\begin{example}
The solution of the fractional differential equation
  $$x^{(\alpha)}(t)=f(t), \quad x(0)=x_0$$ is $$
  x(t)=x_0 +\Gamma^{-1}(\alpha)\int_0^t(t-\tau)^{\alpha-1}f(\tau)d\tau.$$
\end{example}
One can easily generalize the previous definitions and results for
functions with a domain $[a,b]$:
$$f^{(\alpha)}(t)=\frac{1}{\Gamma(1-\alpha)}\frac{d}{dt}\int_a^t(t-\tau)^{-\alpha}(f(\tau)-f(a))\,d\tau$$
and
$$\int_a^tf(\tau)(d\tau)^\alpha=\alpha\int_a^t(t-\tau)^{\alpha-1}f(\tau)d\tau.$$
For the discussion to follow, we will need the following formula of
integration by parts:
\begin{equation}\label{int:parts}
\int_a^bu^{(\alpha)}(t)v(t)\, (dt)^\alpha =\alpha! [u(t)v(t)]_a^b
-\int_a^bu(t)v^{(\alpha)}(t)\, (dt)^\alpha,
\end{equation}
where $\alpha!:=\Gamma(1+\alpha)$.

%----------------------------------------------------------------------------

\section{Main Results}\label{sec2}
Let us consider the functional defined by
\begin{equation*}
\mathcal{J}(y)=\int_a^b L(x,y(x),y^{(\alpha)}(x),y(a),y(b)) \,
(dx)^{\alpha},
\end{equation*}
where $L(\cdot,\cdot,\cdot,\cdot,\cdot) \in
C^1([a,b]\times\mathbb{R}^4; \mathbb{R})$ and $x\rightarrow \partial_3L(t)$ has continuous $\alpha$-derivative. The fractional problem of
the calculus of variations under consideration has the form
\begin{equation*}
\mathcal{J}(y) \longrightarrow \text{extr}
\end{equation*}
\begin{equation}\label{Funct1}
(y(a)=y_{a}), \quad (y(b)=y_{b})
\end{equation}
\begin{equation*}
y(\cdot)\in C^0.
\end{equation*}
Using parentheses around the end-point conditions means that the
conditions may or may not be present. \\
Along the work we denote by $\partial_iL$, $i=1,\ldots,5$, the
partial derivative of function $L(\cdot,\cdot,\cdot,\cdot,\cdot)$
with respect to its $i$th argument. \\
The following lemma will be needed in the next subsection.
\begin{lemma}
\label{fundLemma} Let $g$ be a continuous function and assume that
$$\int_a^bg(x)h(x)\, (dx)^\alpha=0$$
for every continuous function $h$ satisfying $h(a)=h(b)=0$. Then $g
\equiv0$.
\end{lemma}
\begin{proof}
Can be done in a similar way as the proof of the standard
fundamental lemma of the calculus of variations (see, \textrm{e.g.},
\cite{Brunt}).
\end{proof}
%---------------------------------------------------------------------------
\subsection{Necessary Conditions}

Next theorem gives necessary optimality conditions for the problem
(\ref{Funct1}).

\begin{theorem}
\label{Theo E-L1} Let $y$ be an extremizer to problem
(\ref{Funct1}). Then, $y$ satisfies the fractional Euler--Lagrange
equation
\begin{equation}
\label{E-L1}
\partial_2L(x,y(x),y^{(\alpha)}(x),y(a),y(b))
=\frac{d^{\alpha}}{dx^{\alpha}}
\partial_3L(x,y(x),y^{(\alpha)}(x),y(a),y(b))
\end{equation}
for all $x\in[a,b]$. Moreover, if $y(a)$ is not specified, then
\begin{equation}\label{new:bca}
\int_a^b \partial_4L(x,y(x),y^{(\alpha)}(x),y(a),y(b)) \, (dx)^{\alpha}
=\alpha!\partial_3L(a,y(a),y^{(\alpha)}(a),y(a),y(b))
\end{equation}
if $y(b)$ is not specified, then
\begin{equation}\label{new:bcb}
\int_a^b \partial_5L(x,y(x),y^{(\alpha)}(x),y(a),y(b)) \, (dx)^{\alpha}
=-\alpha!\partial_3L(b,y(b),y^{(\alpha)}(b),y(a),y(b)).
\end{equation}
\end{theorem}

\begin{proof}
Suppose that $y$ is an extremizer of $\mathcal{J}$ and consider the
value of $\mathcal{J}$ at a nearby function $\tilde{y}=y +
\varepsilon h$, where $\varepsilon\in \R$ is a small parameter and
$h$ is an arbitrary continuous function. We do not require $h(a)=0$
or $h(b)=0$ in case $y(a)$ or $y(b)$, respectively, is free (it is
possible that both are free). Let $j(\varepsilon)=\mathcal{J}(y +
\varepsilon h) $. Then a necessary condition for $y$ to be an
extremizer is given by $j'(0)=0$. Hence,
\begin{equation}
\label{eq:FT} \int_a^b \Bigl[
\partial_2L(\cdot)h(x)+\partial_3L(\cdot)h^{(\alpha)}(x)
+\partial_4L(\cdot)h(a)+\partial_5L(\cdot)h(b)\Bigl](dx)^{\alpha} =0
\, ,
\end{equation}
where $(\cdot) = \left(x,y(x),y^{(\alpha)}(x), y(a),y(b)\right)$.
Using integration by parts (\ref{int:parts}) to the second term we get
\begin{multline}\label{eq:aft:IP}
\int_a^b \Bigl[
\partial_2L(\cdot)-\frac{d^{\alpha}}{dx^{\alpha}}
\partial_3L(\cdot)\Bigl](dx)^{\alpha}
+\left.\left.\alpha!\partial_3L(\cdot)\right|_{x=b}h(b)-\alpha!\partial_3L(\cdot)\right|_{x=a}h(a)\\
+\int_a^b
\Bigl[\partial_4L(\cdot)h(a)+\partial_5L(\cdot)h(b)\Bigl](dx)^{\alpha}=0.
\end{multline}
We first consider functions $h$ such that $h(a) =h(b)=0$. Then, by
the Lemma~\ref{fundLemma} we deduce that
\begin{equation*}
\partial_2L(\cdot)
=\frac{d^{\alpha}}{dx^{\alpha}}
\partial_3L(\cdot)
\end{equation*}
for all $x\in[a,b]$. Therefore, in order for $y$ to be an extremizer
to the problem (\ref{Funct1}), $y$ must be a solution of the
fractional Euler--Lagrange equation (\ref{E-L1}). But if $y$ is a solution of
(\ref{E-L1}), the first integral in expression (\ref{eq:aft:IP})
vanishes, and then condition (\ref{eq:FT}) takes the form
\begin{equation*}
h(b)\Bigl[\int_a^b\partial_5L(\cdot)(dx)^{\alpha}+\alpha!\partial_3L(\cdot)|_{x=b}\Bigl]
+h(a)\Bigl[\int_a^b
\partial_4L(\cdot)dx-\alpha!\partial_3L(\cdot)|_{x=a}\Bigl]=0.
\end{equation*}
If $y(a)=y_{a}$ and $y(b)=y_{b}$ are given in the formulation of
problem (\ref{Funct1}), then the latter equation is trivially
satisfied since $h(a)=h(b)=0$. When $y(b)$ is free, then equation
(\ref{new:bcb}) holds, when $y(a)$ is free, then (\ref{new:bca})
holds, since $h(a)$ or $h(b)$ is, respectively, arbitrary.
\end{proof}

In the case $L$ does not depend on $y(a)$ and $y(b)$, by
Theorem~\ref{Theo E-L1} we obtain the following result.

\begin{corollary}\cite[Theorem~1]{Almeida}
Let $y$ be an extremizer to problem
\begin{equation*}
 \mathcal{J}(y)=\int_a^b L(x,y(x),y^{(\alpha)}(x) )(dx)^{\alpha} \longrightarrow \text{extr}.\\
\end{equation*}
Then, $y$ satisfies the fractional Euler--Lagrange equation
\begin{equation*}
\partial_2L(x,y(x),y^{(\alpha)}(x))=\frac{d^{\alpha}}{dx^{\alpha}}
\partial_3L(x,y(x),y^{(\alpha)}(x))
\end{equation*}
for all $x\in[a,b]$. Moreover, if $y(a)$ is not specified, then
\begin{equation*}
\partial_3L(a,y(a),y^{(\alpha)}(a))=0
\end{equation*}
if $y(b)$ is not specified, then
\begin{equation*}
\partial_3L(b,y(b),y^{(\alpha)}(b))=0
\end{equation*}
\end{corollary}

Observe that if $\alpha$ goes to $1$, then the operators
$\frac{d^{\alpha}}{dx^{\alpha}}$, $(dx)^{\alpha}$ could be replaced
with $\frac{d}{dx}$ and $dx$. Thus, in this case we obtain the
corresponding result in the classical context of the calculus of
variations (see \cite[Corollary~1]{MalTor},
\cite[Theorem~2.1]{Cruz}).

\begin{corollary}\cite[Corollary~1]{MalTor}
If $y$ is a local extremizer for
\begin{equation*}
\begin{gathered}
\mathcal{J}(y)=\int_a^b L(x,y(x),
y'(x), y(a),y(b)) \, dx \longrightarrow \text{extr}\\
\quad (y(a)=y_{a}), \quad (y(b)=y_{b}),
\end{gathered}
\end{equation*}
then
\begin{equation*}
\frac{d}{dx}\partial_{3}L(x,y(x), y'(x), y(a),y(b))=
\partial_2L(x,y(x), y'(x), y(a),y(b))
\end{equation*}
for all $x \in [a,b]$. Moreover, if $y(a)$ is free, then
\begin{equation*}
\partial_{3}L(a,y(a), y'(a), y(a),y(b))
=\int_{a}^{b}\partial_{5}L(x,y(x), y'(x), y(a),y(b))dx;
\end{equation*}
and if $y(b)$ is free, then
\begin{equation*}
\partial_{3}L(b,y(b), y'(b), y(a),y(b))
=-\int_{a}^{b}\partial_{6}L(x,y(x), y'(x), y(a),y(b))dx.
\end{equation*}
\end{corollary}

%-----------------------------------------------------------------------------------
\subsection{Sufficient Conditions}
In this section we prove sufficient conditions for optimality. Similarly to what happens in the
classical calculus of variations, some conditions of convexity
(concavity) are in order.

\begin{definition}
Given a function $L$, we say that $L(\underline x,y,z,t,u)$ is
jointly convex (concave) in $(y,z,t,u)$, if $\partial_i L$ ,
$i=2,\ldots,5$, exist and are continuous and verify the following
condition:
\begin{multline*}
L(x,y+y_1,z+z_1,t+t_1,u+u_1)-L(x,y,z,t,u)\\
\geq (\leq) \partial_2 L(\cdot)y_1+\partial_3 L(\cdot)z_1+
\partial_4 L(\cdot)t_1+\partial_5 L(\cdot)u_1,
\end{multline*}
where $(\cdot)=(x,y,z,t,u)$, for all $(x,y,z,t,u)$,
$(x,y+y_1,z+z_1,t+t_1,u+u_1)$ $\in [a,b]\times\mathbb R^4$.
\end{definition}

\begin{theorem}\label{suff}
Let $L(\underline x,y,z,t,u)$ be a jointly convex (concave) in
$(y,z,t,u)$. If $y_0$ satisfies conditions
(\ref{E-L1})--(\ref{new:bcb}), then $y_0$ is a global minimizer
(maximizer) to problem (\ref{Funct1}).
\end{theorem}

\begin{proof}
We shall give the proof for the convex case. Since $L$ is jointly
convex in $(y,z,t,u)$ for any continuous function $y_0+h$, we have
\begin{multline*}
\mathcal{J}(y_0+h)-\mathcal{J}(y_0)
= \int_a^b \Bigl[L(x,y_0(x) + h(x),y^{(\alpha)}_0(x) +
h^{(\alpha)}(x),y_0(a) + h(a),\Bigl.\\
y_0(b) + h(b))\Bigl.-L(x,y_0(x),y^{(\alpha)}_0(x),y_0(a),y_0(b))\Bigl](dx)^{\alpha}\\
\geq \int_a^b \Bigl[
\partial_2L(\cdot)h(x)+\partial_3L(\cdot)h^{(\alpha)}(x)+\partial_4L(\cdot)h(a)\Bigl.
\Bigl.+\partial_5L(\cdot)h(b)\Bigl](dx)^{\alpha}
\end{multline*}
where $(\cdot) = \left(x,y_0(x),y^{(\alpha)}_0(x),
y_0(a),y_0(b)\right)$. We can now proceed analogously to the proof
of Theorem~\ref{Theo E-L1}. As the result we get
\begin{multline*}
\mathcal{J}(y_0+h)-\mathcal{J}(y_0)
\geq \int_a^b \Bigl[
\partial_2L(\cdot)-\frac{d^{\alpha}}{dx^{\alpha}}
\partial_3L(\cdot)\Bigl](dx)^{\alpha}\\
h(b)\Bigl[\int_a^b\partial_5L(\cdot)(dx)^{\alpha}+\alpha!\partial_3L(\cdot)|_{x=b}\Bigl]
+h(a)\Bigl[\int_a^b
\partial_4L(\cdot)dx-\alpha!\partial_3L(\cdot)|_{x=a}\Bigl]=0,
\end{multline*}
since $y_0$ satisfies conditions (\ref{E-L1})--(\ref{new:bcb}).
Therefore, we obtain $\mathcal{J}(y_0+h)\geq \mathcal{J}(y_0)$.
\end{proof}
%------------------------------------------------------------------
\subsection{Examples}
We shall provide examples in order to illustrate our main results.

\begin{example}\label{ex}
Consider the following problem
\begin{multline*}
\mathcal{J}(y)=\int_0^1\left\{\Bigl[
\frac{x^\alpha}{\Gamma(\alpha+1)}(y^{(\alpha)})^2-2x^\alpha
y^{(\alpha)}\Bigl]^2\right.\\
\left.+(y(0)-1)^2+(y(1)-2)^2\right\}\,(dx)^\alpha \longrightarrow
\text{extr}.
\end{multline*}
The Euler--Lagrange equation associated to this problem is
\begin{equation}\label{ELexample}
\frac{d^\alpha}{dx^\alpha}\left(2\left[ \frac{x^\alpha}{\Gamma(\alpha+1)}(y^{(\alpha)})^2-2x^\alpha y^{(\alpha)} \right]\right.
\left.\cdot \left[
\frac{2x^\alpha}{\Gamma(\alpha+1)}y^{(\alpha)}-2x^\alpha
\right]\right)=0.
\end{equation}
Let $y=x^\alpha+b$, where $b\in \R$. Since
$y^{(\alpha)}=\Gamma(\alpha+1)$, it follows that $y$ is a solution
of (\ref{ELexample}). In order to determine $b$ we use the
generalized natural boundary conditions
(\ref{new:bca})--(\ref{new:bcb}), which can be written for this
problem as
\begin{equation*}
\int_0^1 (y(0)-1)(dx)^{\alpha}=0,
\end{equation*}
\begin{equation*}
\int_0^1(y(1)-2)(dx)^{\alpha}=0.
\end{equation*}
Hence, $\tilde{y}=x^\alpha+1$ is a candidate solution. We remark
that the $\tilde{y}$ is not differentiable in $[0,1]$.
\end{example}

\begin{example}
Consider the following problem
\begin{equation}
\label{EX} \mathcal{J}(y)=\int_0^1 \left[(y^{(\alpha)}(x))^2+\gamma
y^2(0)\right.
+\left.\lambda(y(1)-1)^2\right](dx)^{\alpha}\longrightarrow
\text{min},
\end{equation}
where $\gamma,\lambda\in \R^+$. For this problem, the fractional
Euler--Lagrange equation and the generalized natural boundary
conditions (see Theorem~\ref{Theo E-L1}) are given, respectively, as
\begin{equation}\label{Ex:el}
2\frac{d^{\alpha}}{dx^{\alpha}}y^{(\alpha)}(x)=0,
\end{equation}
\begin{equation}\label{Ex:ba}
\int_0^1\gamma y(0)(dx)^{\alpha}=\alpha !y^{(\alpha)}(0),
\end{equation}
\begin{equation}\label{Ex:bb}
\int_0^1\lambda (y(1)-1)(dx)^{\alpha}=-\alpha !y^{(\alpha)}(1).
\end{equation}
Solving equations (\ref{Ex:el})--(\ref{Ex:bb}) we obtain that
\begin{equation*}
\bar{y}(x)=\frac{\gamma \lambda\alpha!}{\gamma \lambda
+(\alpha!)^2(\lambda
+\gamma)}x^{\alpha}+\frac{(\alpha!)^2\lambda}{\gamma \lambda
+(\alpha!)^2(\lambda +\gamma)}
\end{equation*}
is a candidate for minimizer. Observe that problem (\ref{EX})
satisfies assumptions of Theorem~\ref{suff}. Therefore $\bar{y}$ is
a global minimizer to this problem.
We note that when $\alpha$ goes
to $1$ problem (\ref{EX}) tends to
\begin{equation*}
 \mathcal{K}(y)=\int_0^1 \left[(y'(x))^2+\gamma
y^2(0)\right.\\
\left.+\lambda(y(1)-1)^2\right]dx\longrightarrow \min.
\end{equation*}
with the solution
\begin{equation*}
y(x)=\frac{\gamma \lambda}{\gamma \lambda +\lambda
+\gamma}x+\frac{\lambda}{\gamma \lambda +\lambda +\gamma}.
\end{equation*}
\end{example}

%----------------------------------------------------------------

\section{Conclusion}\label{sec3}

In recent years fractional calculus has played an important role in
various fields such as mechanics, electricity, chemistry, biology,
economics, modeling, identification, control theory and signal
processing (see, \textrm{e.g.},
\cite{Ric:Del:Hold,Baleanu5,Debnath,Kolwankar,Machado,Ortigueira,Ross}).
The fractional operators are non-local, therefore they are suitable
for constructing models possessing memory. This gives several
possible applications of the FCV, \textrm{e.g.}, in describing
non-local properties of physical systems in mechanics (see,
\textrm{e.g.}, \cite{Carpinteri,Klimek,Rabei2}) or
electrodynamics (see, \textrm{e.g.}, \cite{Baleanu3,Tarasov}). The
Jumarie fractional derivative is quite suitable to describe dynamics
evolving in a space which exhibit coarse-grained phenomenon. When
the point in this space is not infinitely thin but rather a
thickness, then it would be better to replace $dx$ by $(dx)^\alpha$,
$0 < \alpha < 1$, where $\alpha$ characterizes the grade of the
phenomenon. The fractal feature of the space is transported on time,
and so both space and time are fractal. Thus, the increment of time
of the dynamics of the system is not $dx$ but $(dx)^\alpha$. In this
note we generalize some previous results of the FCV (which are
defined in terms of the Jumarie fractional derivatives and
integrals) by proving optimality conditions for problems of FCV with
a Lagrangian density depending on the free end-points. The advantage
of using the Jumarie fractional derivative lies in the fact that
this derivative is defined for continuous functions, non-differentiable (see, Example~\ref{ex}). Note that the integrand in
problem (\ref{Funct1}) depends upon the a priori unknown final
values $y(a)$ and $y(b)$. The present paper indicates
how such problems may be solved.

%------------------------------------------------------------------------

\begin{acknowledgement}
This work is partially supported by the \emph{Portuguese Foundation
for Science and Technology} (FCT) through the \emph{Systems and
Control Group} of the R\&D Unit CIDMA, and partially by BUT Grant S/WI/00/2011.
The author is grateful to Delfim F. M. Torres for
inspiring discussions and useful comments.
\end{acknowledgement}

%-----------------------------------------------------------------------

\end{document}